\documentclass[reqno,12pt]{amsart}

\usepackage{amsmath, epsfig, cite}
\usepackage{amssymb}
\usepackage{amsfonts}
\usepackage{latexsym}
\usepackage{amsthm}

\newtheorem{theorem}{Theorem}

\newtheorem{conjecture}{Conjecture}
\newtheorem{lemma}{Lemma}

\theoremstyle{remark}

\numberwithin{equation}{section}

\allowdisplaybreaks

\author{Yudong Liu and Xiaoxia Wang$^*$}
\address{Department of Mathematics\\
    Shanghai University \\
    Shanghai 200444, P.\:R.\:China}
\email{lydshdx@163.com (Y. Liu), xiaoxiawang@shu.edu.cn (X. Wang)}
\thanks{This work is supported by National Natural Science Foundations of China (11661032).}

\title[$q$-analogues of the  (G.2) Supercongruence of Van Hamme]{$q$-analogues of the  (G.2) Supercongruence of Van Hamme}



\begin{document}

\begin{abstract}
Motivated by the recent research of congruences and $q$-congruences, we provide two different
$q$-analogues of the (G.2) supercongruence of Van Hamme through the `creative microscoping' method, which was devised by Guo and Zudilin. It is a remarkable fact that this is the first time to give direct $q$-analogues of (G.2). In addition, we propose a conjecture related to Swisher's Dwork-type supercongruence (G.3).

\end{abstract}

\maketitle

\section{Introduction}
\noindent In Ramanujan's first letter to Hardy in 1913, he announced that
(cf.\ \cite[p.~25, Equation~(2)]{BR})
\begin{align}
\sum_{k=0}^\infty(8k+1)\frac{(\frac{1}{4})_k^4}{k!^4}
=\frac{2\sqrt{2}}{\sqrt{\pi}\,\Gamma(\frac 34)^2} , \label{eq:4ram}
\end{align}
along with some similar hypergeometric identities, but he did not give any proofs.
Here  $(a)_n=a(a+1)\cdots(a+n-1)$ denotes the Pochhammer symbol and $\Gamma(x) $ is the Gamma function. The identity \eqref{eq:4ram} was ultimately proved by Hardy in \cite[p.~495]{Ha}.
In 1997, Van Hamme\cite{Hamme} proposed 13 mysterious $p$-adic analogues of Ramanujan-type $\pi$-formulas, such as,
\begin{align}
(\text{G.2})\quad\quad&\sum_{k=0}^{(p-1)/4}(8k+1)\frac{(\frac{1}{4})_k^4}{k!^4}
\equiv p\frac{\Gamma_p(\frac 12)\Gamma_p(\frac 14)}{\Gamma_p(\frac 34)}
\pmod{p^3}  \quad\quad p\equiv 1 \pmod 4.
\label{eq:pram}
\end{align}
Here and throughout this paper, $p$ is an odd prime and $\Gamma_p(x)$ is the $p$-adic Gamma
function~\cite{Mor}. Van Hamme\cite{Hamme} himself proved (C.2), (H.2) and (I.2). Later, Swisher\cite{Swisher} proved that the supercongruence \eqref{eq:pram} is true modulo $p^4$ for $p\equiv 1\pmod4$.

During the past few years, the Ramanujan-type congruences and supercongruences, which are viewed as the $p$-adic analogues of Ramanujan-type formulas,
have  caught attention of many authors (see \cite{GW,Guo2,Guo3,Guo4,Guo5,Guo6,Guo-t,Guo20,GuoZu,Guo21,WY0,WY,Zud2009,Long}).
Among them, Guo \cite{Guo2,Guo3,Guo4,Guo20,Guo6} and Guo and Wang \cite{GW} gave $q$-analogues of most of
Van Hamme's 13 conjectural supercongruences by using the $q$-WZ method. Guo and Zudilin \cite{GuoZu} introduced  the `creative microscoping' method
to prove and reprove many $q$-congruences. Wang and Yue\cite {WY} succeeded in proving a $q$-analogue of Van Hamme's supercongruence (A.2) for any prime $p\equiv 3 \pmod 4$.
A $q$-analogue of (A.2) for primes $p\equiv 1 \pmod 4$ was then given by Guo\cite{Guo5}.
However, no $q$-analogues of Van Hamme's (G.2) supercongruence have been found so far.

Recently, Guo and Schlosser  \cite[Theorems 2]{GS20} proved that, for even $d \ge 4$ and positive integer $n$ with $n\equiv-1 \pmod d$,
\begin{align}
\sum_{k=0}^{n-1}[2dk+1]\frac{(q;q^d)_k^d}{(q^d;q^d)_k^d}q^{\frac{d(d-3)k}{2}}
\equiv 0\pmod{\Phi_n(q)^2}, \label{eq:Guo1}
\end{align}
which is a $q$-analogue of the $p$-adic analogue of (\ref{eq:4ram}) for $p\equiv 3 \pmod 4$ when $d=4$.
Moreover, some other interesting $q$-congruences can be found in \cite{LP,NP,Tauraso2,Zudilin}.

In this paper, we shall give two different  $q$-analogues of the (G.2) supercongruence of Van Hamme.
\begin{theorem}\label{thm:1}
Let $n\equiv 1\pmod {4}$ be a positive integer. Then
\begin{align}
&\sum_{k=0}^{(n-1) / 4}[8k+1] \frac{\left(q ; q^{4}\right)_{k}^4}{\left(q^4 ; q^{4}\right)_k^4}q^{2k} \equiv \frac{\left(q^{2} ; q^{4}\right)_{(n-1)/{4}}}{\left(q^{4} ; q^{4}\right)_{(n-1)/4}}[n] q^{(1-n) / 4}  \pmod {[n]\Phi_{n}(q)^{2}}\label{eq;thm1_1};\\
&\sum_{k=0}^{n-1}[8k+1] \frac{\left(q ; q^{4}\right)_{k}^4}{\left(q^4 ; q^{4}\right)_k^4}q^{2k} \equiv \frac{\left(q^{2} ; q^{4}\right)_{(n-1)/{4}}}{\left(q^{4} ; q^{4}\right)_{(n-1)/4}}[n] q^{(1-n) / 4} \pmod {[n]\Phi_{n}(q)^{2}}.\label{eq;thm1_2}
\end{align}
\end{theorem}
In fact, setting $n=p \equiv 1 \pmod 4$ and $q \rightarrow 1$ in $(\ref{eq;thm1_1})$, we get
\begin{equation}
\sum_{k=0}^{(p-1)/4}(8k+1)\frac{(\frac{1}{4})_k^4}{k!^4}
\equiv \frac{\left(\frac{1}{2}\right)_{(p-1) / 4}}{(1)_{(p-1) / 4}} p
\pmod {p^3}\label{eq;thm1_3}.
\end{equation}
For prime $p\ge 5$, the $p$-adic Gamma function $\Gamma_{p}$ has the following basic properties \cite{Long},
\begin{equation*}
\Gamma_{p}(1)=-1,\quad \Gamma_{p}(\frac{1}{2})^2=(-1)^{\frac{p+1}{2}},\quad
(a)_n=(-1)^n\frac{\Gamma_{p}(a+n)}{\Gamma_{p}(a)},
\end{equation*}
\begin{equation*}
\Gamma_{p}(a+bp)\equiv \Gamma_{p}(a)(1+G_1(a)bp) \pmod{p^2}, \quad
G_1(a)=G_1(1-a),
\end{equation*}
 where $G_{1}(a):=\Gamma_{p}'(a) / \Gamma_{p}(a)$.
 Then we can rewrite the right-hand side of $(\ref{eq;thm1_3})$ as
 \begin{align*}
 \frac{\left(\frac{1}{2}\right)_{(p-1) / 4}}{(1)_{(p-1) / 4}} p
 =\frac{\Gamma_{p}(1)\Gamma_{p}(\frac{1}{4}+\frac{p}{4})}{\Gamma_{p}(\frac{1}{2}) \Gamma_{p}(\frac{3}{4}+\frac{p}{4})} p
 &\equiv-\frac{\Gamma_{p}(\frac{1}{4})(1+G_1(\frac{1}{4})\frac{p}{4})}{\Gamma_{p}(\frac{1}{2}) \Gamma_{p}(\frac{3}{4})(1+G_1(\frac{3}{4})\frac{p}{4})} p \pmod{p^3}\\
 &=\frac{\Gamma_p(\frac 12)\Gamma_p(\frac 14)}{\Gamma_p(\frac 34)}p \pmod{p^3},
 \end{align*}
 which is just the right-hand side of Van Hamme's (G.2) supercongruence.

 Likewise, we have the following supercongruence as $q\to 1$ in \eqref{eq;thm1_2}:
 \begin{align*}
 \sum_{k=0}^{p-1}(8k+1)\frac{(\frac{1}{4})_k^4}{k!^4}
 \equiv p\frac{\Gamma_p(\frac 12)\Gamma_p(\frac 14)}{\Gamma_p(\frac 34)}
 \pmod{p^3}\quad\text{ $p\equiv 1\pmod{4}$},
 \end{align*}
 which is an equivalent form of \eqref{eq:pram}, since $(\frac{1}{4})_k/k!\equiv 0\pmod{p}$ for $(p-1)/4<k\leq  p-1$.


 \begin{theorem}\label{thm:2}
Let $n\equiv 1\pmod {4}$ be a positive integer. Then, modulo $ [n]_{q^2} \Phi_{n}(q^2)^2$,
\begin{align}
\sum_{k=0}^{(n-1) / 4}[8k+1]_{q^{2}}[8k+1]^{2} \frac{\left(q^{2} ; q^{8}\right)_{k}^{4}}{\left(q^{8} ; q^{8}\right)_{k}^{4}} q^{-4k }
&\equiv -\frac{2[n]_{q^{2}}(q^{4};q^{8})_{(n-1)/4}}{(1+q^{2})(q^{8};q^{8})_{(n-1)/4}}q^{(3-n)/2} ,\label{eq;thm2_1}  \\
\sum_{k=0}^{n-1}[8k+1]_{q^{2}}[8k+1]^{2} \frac{\left(q^{2} ; q^{8}\right)_{k}^{4}}{\left(q^{8} ; q^{8}\right)_{k}^{4}} q^{-4 k}
&\equiv -\frac{2[n]_{q^{2}}(q^{4};q^{8})_{(n-1)/4}}{(1+q^{2})(q^{8};q^{8})_{(n-1)/4}}q^{(3-n)/2}.\label{eq;thm2_2}
\end{align}
\end{theorem}
Leting $n=p$ and $q\rightarrow -1$ in \eqref{eq;thm2_1}, we obtain (G.2) once more. Further,
we have the following similar supercongruences by taking $q\rightarrow 1$ in Theorem \ref{thm:2}:
\begin{align*}
\sum_{k=0}^{(p-1)/4}(8k+1)^3\frac{(\frac{1}{4})_k^4}{k!^4}
&\equiv- \frac{\Gamma_p(\frac 12)\Gamma_p(\frac 14)}{\Gamma_p(\frac 34)}p \pmod{p^3}, \\
\sum_{k=0}^{p-1}(8k+1)^3\frac{(\frac{1}{4})_k^4}{k!^4}
&\equiv -\frac{\Gamma_p(\frac 12)\Gamma_p(\frac 14)}{\Gamma_p(\frac 34)}p \pmod{p^3}.
\end{align*}

As for prerequisites, the reader is expected to know the standard $q$-notation.
For an indeterminate $q$,
$(a;q)_n=(1-a)(1-aq)\cdots (1-aq^{n-1})$
is called the {\em $q$-shifted factorial}.
For convenience, we compactly write
$(a_1,a_2,\ldots,a_m;q)_n=(a_1;q)_n (a_2;q)_n\cdots (a_m;q)_n$
for the product of $q$-shifted factorials.
Moreover, $\Phi_n(q)$ denotes the $n$-th {\em cyclotomic polynomial} in $q$, which is defined as
\begin{align*}
\Phi_n(q)=\prod_{\substack{1\leqslant k\leqslant n\\ \gcd(n,k)=1}}(q-\zeta^k),
\end{align*}
where $\zeta$ is an $n$-th primitive root of unity. Furthermore, for arbitrary integer $n$,
$[n]=[n]_q=(1-q^n)/(1-q)$ is the {\em $q$-integer}.

The rest of the paper is organized as follows. 
We  shall prove Theorems \ref{thm:1} and \ref{thm:2}  based on Rogers' nonterminating $_6\phi_5$ summation and Watson's $_8\phi_7$ transformation in the Sections 2 and 3.
Certain generalizations of Theorems \ref{thm:1}  and \ref{thm:2} will be given in Section 4.
Finally, in Section 5, we will propose a $q$-analogue of Swisher's Dwork-type conjecture supercongruence (G.3) with $p\equiv1 \pmod 4$.


\section{Proof of Theorem 1 }\label{sec:thm1}
We start with Rogers' nonterminating $_6\phi_5$ summation (cf.\
\cite[Appendix~(II.20)]{GR}):
\begin{equation}\label{eq:6phi5}
{}_{6}\phi_5\!\left[\begin{array}{cccccc}
a,& qa^{\frac{1}{2}},& -qa^{\frac{1}{2}}, & b,    & c,    & d \\
& a^{\frac{1}{2}}, & -a^{\frac{1}{2}},  & aq/b, & aq/c, & aq/d
\end{array};q,\, \frac{aq}{bcd}
\right]
=\frac{(aq, aq/bc,aq/bd,aq/cd;q)_\infty}
{(aq/b,aq/c,aq/d,aq/bcd;q)_\infty} ,
\end{equation}
where $|aq/bcd|<1$ for convergence.

Also, the following lemmas are needed in our proof.

\begin{lemma}\label{lem:one}
Let $d\ge 2$, $m>1$,  $0\le s \le m-1$, $t$ be integers with $\gcd(d,t)=1$ and $ds\equiv -t \pmod m$. Then, for $0\le k\le s$, we have
\begin{equation}
\frac{\left(a q^{ t} ; q^{ d}\right)_{s-k}}{\left(q^{ d} / a ; q^{ d}\right)_{s-k}} \equiv(-a)^{s-2 k} q^{s(d s-d+2 t)/2+(d-t) k} \frac{\left(a q^{ t} ; q^{ d}\right)_{k}}{\left(q^{ d} / a ; q^{ d}\right)_{k}} \quad \pmod {\Phi_{m}\left(q\right)}\label{eq;Lemma1}.
\end{equation}

\end{lemma}
\begin{proof}
Since $q^m\equiv 1 \pmod {\Phi_{m}(q)}$, we have
\begin{align*}
\begin{aligned}
\frac{\left(a q^{ t} ; q^{ d}\right)_{s}}{\left(q^{ d} / a ; q^{ d}\right)_{s}} &=\frac{\left(1-a q^{ t}\right)\left(1-a q^{ t+ d}\right) \cdots\left(1-a q^{t+d s- d}\right)}{\left(1-q^{ d} / a\right)\left(1-q^{2 d} / a\right) \cdots\left(1-q^{ d s} / a\right)} \\
& \equiv \frac{\left(1-a q^{t}\right)\left(1-a q^{t+d }\right) \cdots\left(1-a q^{t+ d s- d}\right)}{\left(1-q^{ d- d s- t} / a\right)\left(1-q^{2 d- d s- t} / a\right) \cdots\left(1-q^{- t} / a\right)} \\
&=(-a)^{s} q^{s(2t+d s-d)/2}\quad \pmod {\Phi_{m}(q)}.
\end{aligned}
\end{align*}
 For $0\le k \le s$, we obtain
\begin{align*}
\frac{\left(a q^{ t} ; q^{ d}\right)_{s-k}}{\left(q^{ d} / a ; q^{ d}\right)_{s-k}}
&=\frac{\left(a q^{t} ; q^{ d}\right)_{s}}{\left(q^{ d} / a ; q^{ d}\right)_{s}} \frac{\left(1-q^{d s- (k-1) d} / a\right) \cdots\left(1-q^{ d s} / a\right)}{\left(1-a q^{ d s-d k+ t}\right) \cdots\left(1-a q^{ d s- d+ t}\right)}\\
&\equiv\frac{\left(a q^{ t} ; q^{d}\right)_{s}}{\left(q^{ d} / a ; q^{ d}\right)_{s}} \frac{\left(1-q^{-d k+ d- t} / a\right) \cdots\left(1-q^{- t} / a\right)}{\left(1-a q^{- d k}\right) \cdots\left(1-a q^{- d}\right)}\quad \pmod {\Phi_{m}(q)}\\
&\equiv(-a)^{s-2 k} q^{s(d s-d+2 t)/2+(d-t) k} \frac{\left(a q^{ t} ; q^{ d}\right)_{k}}{\left(q^{ d} / a ; q^{ d}\right)_{k}} \quad \pmod {\Phi_{m}(q)}
\end{align*}
as desired.
\end{proof}

\begin{lemma}\label{lem:two}
Let $m>1$, $d\ge 2$, $t$ be integers with $\gcd(d,m)=1$  and $\gcd(d,t)=1$. Then
\begin{align}
\sum_{k=0}^{m-1}[2dk+t] \frac{\left(q^t ; q^{d}\right)_{k}^2(aq^t;q^d)_k(q^t/a;q^d)_k}{\left(q^{d }; q^{d}\right)_k^2(aq^d;q^d)_k(q^d/a;q^d)_k}q^{(d-2t)k} \equiv 0 \pmod {\Phi_{m}(q)}\label{eq;lemma2_1}.
\end{align}
\end{lemma}

\begin{proof}
Since $\gcd(d,m)=1$, there exists a unique integer $s$ with $0\le s\le m-1$ and $ds\equiv -t \pmod m $. Applying Lemma 1, for $0 \le k \le s$, we have
\begin{align*}
&[2d(s-k)+t]\frac{\left(q^t ; q^{d}\right)_{s-k}^2(aq^t;q^d)_{s-k}(q^t/a;q^d)_{s-k}}{\left(q^{d }; q^{d}\right)_{s-k}^2(aq^d;q^d)_{s-k}(q^d/a;q^d)_{s-k}}q^{(d-2t)(s-k)}\\
&\equiv -[2dk+t] \frac{\left(q^t ; q^{d}\right)_{k}^2(aq^t;q^d)_k(q^t/a;q^d)_k}{\left(q^{d }; q^{d}\right)_k^2(aq^d;q^d)_k(q^d/a;q^d)_k}q^{(d-2t)k}  \pmod {\Phi_{m}(q)}.
\end{align*}
Hence, if $s$ is odd, then we get
\begin{align}
\sum_{k=0}^{s}[2dk+t] \frac{\left(q^t ; q^{d}\right)_{k}^2(aq^t;q^d)_k(q^t/a;q^d)_k}{\left(q^{d }; q^{d}\right)_k^2(aq^d;q^d)_k(q^d/a;q^d)_k}q^{(d-2t)k} \equiv 0 \pmod {\Phi_{m}(q)}\label{eq;lemma2_2}.
\end{align}
On the other hand, if $s$ is even, then the middle term of (\ref{eq;lemma2_2}) contains the factor $[2d(\frac{s}{2})+t]=[ds+t]$,
which is congruent to $0$ modulo $\Phi_{m}(q)$. Then we arrive at (\ref{eq;lemma2_2})  for $0\le s \le m-1$.
Furthermore, since $(q^t;q^d)_k/(q^d;q^d)_k\equiv 0 \pmod {\Phi_{m}(q)}$ for $s< k \le m-1$, we directly obtain  (\ref{eq;lemma2_1}).
This completes the proof of the lemma.
\end{proof}

We now present the following parametric generalization of Theorem \ref{thm:1}.
\begin{theorem}\label{thm:3}
Let $n\equiv 1\pmod {4}$ be a positive integer. For any indeterminate $a$,  modulo $[n](1-aq^n)(a-q^n)$, we have
\begin{equation}
\sum_{k=0}^{(n-1)/4}[8k+1] \frac{\left(q ; q^{4}\right)_{k}^2(aq;q^4)_k(q/a;q^4)_k}{\left(q^{4 }; q^{4}\right)_k^2(aq^4;q^4)_k(q^4/a;q^4)_k}q^{2k} \equiv \frac{\left(q^{2} ; q^{4}\right)_{(n-1)/{4}}}{\left(q^{4} ; q^{4}\right)_{(n-1)/4}}[n] q^{(1-n) / 4} \label{eq;thm3_1}.
\end{equation}

\end{theorem}

\begin{proof}
For  $a=q^{n}$ or $a=q^{-n}$, the left-hand side of $(\ref{eq;thm3_1})$ is equal to
\begin{align*}
\sum_{k=0}^{(n-1)/4}[8k+1] \frac{\left(q ; q^{4}\right)_{k}^2(q^{1+n};q^4)_k(q^{1-n};q^4)_k}{\left(q^{4 }; q^{4}\right)_k^2(q^{4+n};q^4)_k(q^{4-n};q^4)_k}q^{2k},
\end{align*}
which by Rogers' summation (\ref{eq:6phi5}) with the parameter substitutions $q \mapsto q^{4}$, $a=d=q$, $b=q^{1-n}$ and $c=q^{1+n}$ can be written as
\begin{align}
&{}_{6}\phi_5\!\left[\begin{array}{cccccc}
q,& q^{\frac{9}{2}},& -q^{\frac{9}{2}}, & q^{1-n},    & q^{1+n},    & q \\
& q^{\frac{1}{2}}, & -q^{\frac{1}{2}},  & q^{4+n}, & q^{4-n}, & q^{4}
\end{array};q^4,\,q^{2}
\right]\notag \\
&=\frac{\left(q^{5}, q^{3}, q^{3-n}, q^{3+n} ; q^{4}\right)_{\infty}}{\left(q^{4-n},  q^{4+n}, q^{4} , q^{2} ; q^{4}\right)_{\infty}} \notag \\
&=\frac{\left(q^{2} ; q^{4}\right)_{(n-1)/{4}}}{\left(q^{4} ; q^{4}\right)_{(n-1)/4}}[n] q^{(1-n) / 4}.
\end{align}
This means that the $q$-congruence (\ref{eq;thm3_1}) holds modulo $1-aq^n$ and $a-q^n$.

In what follows we shall prove
\begin{equation}
\sum_{k=0}^{(n-1)/4}[8k+1] \frac{\left(q ; q^{4}\right)_{k}^2(aq;q^4)_k(q/a;q^4)_k}{\left(q^{4 }; q^{4}\right)_k^2(aq^4;q^4)_k(q^4/a;q^4)_k}q^{2k} \equiv 0 \pmod{[n]} .
\end{equation}
Let $\zeta \not=1 $ be an $n$-th unity root, not necessarily primitive. Then $\zeta$ must be a primitive  $m_1$-th root of unity with $m_1 | n$. Since  $\gcd(m_1,4)=1$,
there exists a unique integer $s_1$ with $0< s_1 \le m_1-1$ and $4s_1\equiv -1 \pmod {m_1}$. Let $c_q(k)$ denote the $k$-th term on the left-hand side in \eqref{eq;thm3_1}, i.e,
\begin{equation*}
c_q(k)=[8k+1]\frac{\left(q ; q^{4}\right)_{k}^2(aq;q^4)_k(q/a;q^4)_k}{\left(q^{4 }; q^{4}\right)_k^2(aq^4;q^4)_k(q^4/a;q^4)_k}q^{2k}.
\end{equation*}
Letting $d=4$, $t=1$, $m=m_1$ in (\ref{eq;lemma2_1}) and combining (\ref{eq;lemma2_2}), we have
\begin{align*}
 \sum_{k=0}^{m_{1}-1} c_{\zeta}(k)=\sum_{k=0}^{s_1} c_{\zeta}(k)=0.
\end{align*}
For $0\le k \le m_1-1$, the following  limit holds:
\begin{equation*}
\lim _{q \rightarrow \zeta} \frac{c_{q}\left(l m_{1}+k\right)}{c_{q}\left(l m_{1}\right)}=c_{\zeta}(k).
\end{equation*}
Then, we obtain
\begin{align}
\begin{aligned}
\sum_{k=0}^{\frac{n-1}{4}} c_{\zeta}(k) = \sum_{l=0}^{\frac{n-4 s_{1}-1}{ 4m_{1}}-1} c_{\zeta}\left(l m_{1}\right) \sum_{k=0}^{m_{1}-1} c_{\zeta}(k)
+c_{\zeta}\left(\left(n-4 s_{1}-1\right) / 4\right) \sum_{k=0}^{s_{1}} c_{\zeta}(k)=0;
\end{aligned}
\end{align}
\begin{align}
\sum_{k=0}^{n-1} c_{\zeta}(k)= \sum_{l=0}^{n / m_{1}-1} c_{\zeta}\left(l m_{1}\right) \sum_{k=0}^{m_{1}-1} c_{\zeta}(k)=0.
\end{align}
It follows that
\begin{align*}
\sum_{k=0}^{M}[8k+1] \frac{\left(q ; q^{4}\right)_{k}^2(aq;q^4)_k(q/a;q^4)_k}{\left(q^{4 }; q^{4}\right)_k^2(aq^4;q^4)_k(q^4/a;q^4)_k}q^{2k} \equiv 0 \pmod{\Phi_{m_1}(q)} ,
\end{align*}
where $M=(n-1)/4$ or $n-1$.
Noting that
\[\prod_{m_1 | n, m_1>1} \Phi_{m_1}(q)=[n],\]
we immediately get
\begin{align}
\sum_{k=0}^{M}[8k+1] \frac{\left(q ; q^{4}\right)_{k}^2(aq;q^4)_k(q/a;q^4)_k}{\left(q^{4 }; q^{4}\right)_k^2(aq^4;q^4)_k(q^4/a;q^4)_k}q^{2k} \equiv 0 \pmod{[n]}\label{eq;integer}.
\end{align}
Since $[n]$, $a - q^n$ and $1 - aq^n$ are pairwise relatively prime polynomials, we complete the proof of the theorem.
\end{proof}
\begin{proof}[Proof of Theorem 1.]
 For $k$ in the range $0\le k \le (n-1)/4$, since $\gcd(n,4)=1$, the numbers $4,8 \cdots 4(n-1)$ are all not divisible by $n$. So that the limit $a \rightarrow 1$
 of the denominator related to $a$ in $(\ref{eq;thm3_1})$  is relatively prime to $\Phi_{n}(q)$. On the other hand, the limit $(1-aq^n)(a-q^n)$ as  $a \rightarrow 1$
 contains the factor $\Phi_{n}(q)^2$. Thus, letting $a\rightarrow 1$ in (\ref{eq;thm3_1}), we conclude that (\ref{eq;thm1_1}) is true modulo $\Phi_{n}(q)^3$.
 Setting $a \rightarrow 1$ in \eqref{eq;integer}, we get
\begin{align}
\sum_{k=0}^{M}[8k+1] \frac{\left(q ; q^{4}\right)_{k}^4}{\left(q^{4 }; q^{4}\right)_k^4}q^{2k} \equiv 0 \pmod{[n]}\label{eq;thm3_2},
\end{align}
which means that (\ref{eq;thm1_1}) also holds modulo $[n]$. Since the least common multiple of $[n]$ and $\Phi_{n}(q)^3$ is $[n]\Phi_{n}(q)^2$, we obtain (\ref{eq;thm1_1}).
Moreover, in view of $(q;q^4)_k^4/(q^4;q^4)_k^4\equiv 0 \pmod {\Phi_{n}(q)^4}$ for $(n-1)/4 <k \le n-1$, we arrive at (\ref{eq;thm1_2}). This completes the proof.
\end{proof}

\section{Proof of Theorem 2 }\label{sec:thm1}
In this section, we need Watson's  $_8\phi_7$ transformation formula (cf.\
\cite[Appendix~(II.17)]{GR})
\begin{align}
& _{8}\phi_{7}\!\left[\begin{array}{cccccccc}
a,& qa^{\frac{1}{2}},& -qa^{\frac{1}{2}}, & b,    & c,    & d,    & e,    & f \\
& a^{\frac{1}{2}}, & -a^{\frac{1}{2}},  & aq/b, & aq/c, & aq/d, & aq/e, & aq/f
\end{array};q,\, \frac{a^2q^{2}}{bcdef}
\right] \notag\\[5pt]
&\quad =\frac{(aq, aq/de,aq/df,aq/ef;q)_\infty}
{(aq/d, aq/e,aq/f,aq/def;q)_\infty}
\,{}_{4}\phi_{3}\!\left[\begin{array}{c}
aq/bc,\ d,\ e,\ f \\
aq/b,\, aq/c,\, def/a
\end{array};q,\, q
\right]\label{eq:8phi7}
\end{align}
to accomplish our proof. Moreover, we require the following lemma.
\begin{lemma}\label{lem:three}
Let $m>1$, $d\ge 2$, $t$ be integers with $\gcd(d,m)=1$  and $\gcd(d,t)=1$. Then
\begin{equation}
\sum_{k=0}^{m-1}[2d k+t]_{q^{2}}[2dk+t]^{2} \frac{\left(q^{2t} ; q^{2d}\right)_{k}^{2}(aq^{2t};q^{2d})_{k}(q^{2t}/a;q^{2d})_k}{\left(q^{2d} ; q^{2d}\right)_{k}^{2}(aq^{2d};q^{2d})_{k}(q^{2d}/a;q^{2d})_k} q^{-4t k}
\equiv 0 \pmod{\Phi_m(q^2)}\label{eq;Lemma3_1}.
\end{equation}
\end{lemma}
\begin{proof}
Setting $q\mapsto q^2$ in (\ref{eq;Lemma1}), we get
\begin{align}\label{eq;Lemma3_2}
\frac{\left(a q^{ 2t} ; q^{ 2d}\right)_{s-k}}{\left(q^{ 2d} / a ; q^{ 2d}\right)_{s-k}}
\equiv(-a)^{s-2 k} q^{s(d s-d+2 t)+2(d-t)k} \frac{\left(a q^{2t} ; q^{ 2d}\right)_{k}}{\left(q^{2 d} / a ; q^{2d}\right)_{k}} \quad\left(\bmod \Phi_{m}\left(q^2\right)\right),
\end{align}
where $0\le s\le m-1$ and $ds\equiv -t \pmod m $. Similarly as the proof of Lemma \ref{lem:two}, by (\ref{eq;Lemma3_2}), we can see that the sum of the $k$-th and $(s-k)$-th terms on the left-hand side of (\ref{eq;Lemma3_1}) are congruent to zero modulo $\Phi_{m}(q^2)$ when $k\neq s/2$.
So the following $q$-congruence is true when $s$ is odd:
\begin{align}
\sum_{k=0}^{s}[2d k+t]_{q^{2}}[2dk+t]^{2} \frac{\left(q^{2t} ; q^{2d}\right)_{k}^{2}(aq^{2t};q^{2d})_{k}(q^{2t}/a;q^{2d})_k}{\left(q^{2d} ; q^{2d}\right)_{k}^{2}(aq^{2d};q^{2d})_{k}(q^{2d}/a;q^{2d})_k} q^{-4t k}\equiv 0 \pmod {\Phi_{m}(q^2)}. \label{eq;Lemma3_3}
\end{align}
On the other hand, if $s$ is even, then $[2d(\frac{s}{2})+t]_{q^2}=[ds+t]_{q^2} \equiv 0 \pmod {\Phi_m(q^2)}$.
This means that \eqref{eq;Lemma3_3} holds for any arbitrary integer $0\le s \le m-1$.
Since $(q^{2t};q^{2d})_k/(q^{2d};q^{2d})_k\equiv 0 \pmod {\Phi_{m}(q^2)}$ for $s< k \le m-1$, we immediately arrive at (\ref{eq;Lemma3_1}).
\end{proof}

In order to prove Theorem \ref{thm:2}, we also need to establish the following parametric generalization.
\begin{theorem}\label{thm:4}
Let $n\equiv 1\pmod {4}$ be a positive integer. Then, for any indeterminate $a$, modulo $[n]_{q^2}(1-aq^{2n})(a-q^{2n})$, we have
\begin{align}
&\sum_{k=0}^{(n-1) / 4}[8 k+1]_{q^{2}}[8k+1]^{2} \frac{\left(q^{2} ; q^{8}\right)_{k}^{2}(aq^{2};q^{8})_{k}(q^{2}/a;q^{8})_k}{\left(q^{8} ; q^{8}\right)_{k}^{2}(aq^{8};q^{8})_{k}(q^{8}/a;q^{8})_k} q^{-4 k}\notag \\
&\equiv [n]_{q^{2}}\frac{(q^{4};q^{8})_{(n-1)/4}}{(q^{8};q^{8})_{(n-1)/4}} q^{-(n-1)/2}\left(1-\frac{\left(1-a q^{2 }\right)\left(1-q^{2 } / a\right)}{\left(1-q\right)^{2}(1+ q^{2 })}\right)\label{eq;thm4_1}.
\end{align}
\end{theorem}
\begin{proof}
For $a=q^{2n}$ or $a=q^{-2n}$, the left-hand side of (\ref{eq;thm4_1}) is equal to
\begin{align}
&\sum_{k=0}^{(n-1) / 4}[8 k+1]_{q^{2}}[8k+1]^{2} \frac{\left(q^{2} ; q^{8}\right)_{k}^{2}(q^{2+2n};q^{8})_{k}(q^{2-2n};q^{8})_k}{\left(q^{8} ; q^{8}\right)_{k}^{2}(q^{8+2n};q^{8})_{k}(q^{8-2n};q^{8})_k} q^{-4 k}\notag \\
&=  {_{8}\phi_7} \!\left[\begin{array}{cccccccc}
q^{2},& q^{9},& -q^{9}, & q^{9},    & q^{9},    & q^{2},    & q^{2-2n},    & q^{2+2n} \\
& q, & -q,  & q, & q, & q^{8}, & q^{8-2n}, & q^{8+2n}
\end{array};q^{8},\, q^{-4}
\right], \notag\\[5pt]\label{eq;thm4_2}
\end{align}
where the $_8\phi_7$ series can be evaluated by Watson's  $_8\phi_7$ transformation \eqref{eq:8phi7} with the parameter substitutions
$q \mapsto q^{8}$, $a=d=q^{2}$, $b=c=q^{9}$, $e=q^{2+2n}$ and $f=q^{2-2n}$ as follows:
\begin{align}
&\frac{(q^{10}, q^{6}, q^{6-2n},q^{6+2n};q^{8})_\infty}
{(q^{8}, q^{4}, q^{8-2n}, q^{8+2n};q^{8})_\infty}
\,{}_{4}\phi_{3}\!\left[\begin{array}{c}
    q^{-8},\ q^{2},\ q^{2+2n},\ q^{2-2n} \\
    q,\, q,\, q^{4}
\end{array};q^{8},\, q^{8}
\right] \notag\\[2mm]
&=[n]_{q^{2}}\frac{(q^{4};q^{8})_{(n-1)/4}}{(q^{8};q^{8})_{(n-1)/4}} q^{-(n-1)/2}\left(1-\frac{\left(1- q^{2 +2n}\right)\left(1-q^{2 -2n}  \right)}{\left(1-q\right)^{2}(1+ q^{2 })}\right).
\end{align}
This means that the $q$-congruence (\ref{eq;thm4_1}) modulo $(1-aq^{2n})(a-q^{2n})$ holds true. Moreover, 
for $n>1$, let $\eta \not= 1$ be an $n$-th unity root, not necessarily primitive. Then $\eta$ must be a primitive $m_2$-th root of unity with $m_2 | n$. Owing to $\gcd(m_2,4)=1$, there exists a unique integer $s_2$ with $0< s_2 \le m_2-1$ and $4s_2\equiv -1 \pmod {m_2}$. Setting $d=4$, $t=1$, $s=s_2$, $m=m_2$ in (\ref{eq;Lemma3_1}) and (\ref{eq;Lemma3_3}) we have
\begin{align*}
\sum_{k=0}^{m_{2}-1} p_{\eta}(k)=\sum_{k=0}^{s_2} p_{\eta}(k)=0 \quad \text{and}\quad \sum_{k=0}^{m_{2}-1} p_{-\eta}(k)=\sum_{k=0}^{s_2} p_{-\eta}(k)=0,
\end{align*}
where $p_{q}(k)$ denotes the $k$-th term on the left-hand side of (\ref{eq;thm4_1}). Also, we can calculate that
\begin{align*}
\lim _{q \rightarrow \eta} \frac{p_{q}\left(l m_{2}+k\right)}{p_{q}\left(l m_{2}\right)}=p_{\eta}(k).
\end{align*}
Likewise, we get the following result
\begin{align*}
&\sum_{k=0}^{(n-1) / 4} p_{\eta}(k)=\sum_{\ell=0}^{\frac{n-4 s_{2}-1}{4 m_{2}}-1} p_{\eta}(\ell m_2) \sum_{k=0}^{m_2-1} p_{\eta}(k)+p_{\eta}((n-4s_2-1) /4 )\sum_{k=0}^{s_2} p_{\eta}(k)=0,\\
&\sum_{k=0}^{n-1} p_{\eta}(k)=\sum_{\ell=0}^{n / m_2-1} \sum_{k=0}^{m_2-1} p_{\eta}(\ell m_2+k)=\sum_{\ell=0}^{n / m_2-1} p_{\eta}(\ell m_2) \sum_{k=0}^{m_2-1} p_{\eta}(k)=0,
\end{align*}
which means that $\Phi_{m_2}(q) $ divides the sums $\sum_{k=0}^{(n-t)/d}p_q(k)$ and $\sum_{k=0}^{n-1}p_q(k)$. Similarly, the two sums are also divisible by $\Phi_{m_2}(-q)$, By the relation
\begin{equation*}
\prod_{m_2 | n, m_2>1}\left(\Phi_{m_2}(q) \Phi_{m_2}(-q)\right)=[n]_{q^{2}},
\end{equation*}
we obtain
\begin{align}
\sum_{k=0}^{M}[8k+1]_{q^{2}}[8k+1]^{2} \frac{\left(q^{2} ; q^{8}\right)_{k}^{2}(aq^{2};q^{8})_{k}(q^{2}/a;q^{8})_k}{\left(q^{8}; q^{8}\right)_{k}^{2}(aq^{8};q^{8})_{k}(q^{8}/a;q^{8})_k} q^{-4 k}
\equiv 0 \pmod{[n]_{q^2}} \label{eq;thm3_3},
\end{align}
where $M=(n-1)/4$ or $n-1$. Since $[n]_{q^2}$, $a - q^{2n}$ and $1 - aq^{2n}$ are pairwise relatively prime polynomials, we complete the proof of the theorem.
\end{proof}

\begin{proof}[Proof of Theorem \ref{thm:2}.]
 As same as the proof of Theorem \ref{thm:1}, letting $a \rightarrow 1$ in \eqref{eq;thm4_1},
we can see that the denominator of (\ref{eq;thm4_1}) is relatively prime to  $\Phi_{n}(q^2)$. On the other hand, $\Phi_{n}(q^2)^2$ is the factor of the
limit of $(1-aq^{2n})(a-q^{2n})$ as $a \to 1$. Thus, we get that (\ref{eq;thm2_1}) holds modulo $\Phi_{n}(q^2)^3$. Meanwhile, letting $a\to 1$ in (\ref{eq;thm3_3}),
we see that (\ref{eq;thm2_1}) is also true modulo $[n]_{q^2}$. Hence, the $q$-supercongruence (\ref{eq;thm2_1}) holds true.
Furthermore, for $(n-1)/4 <k \le n-1$, $(q^2;q^8)_k^4/(q^8;q^8)_k^4\equiv 0 \pmod {\Phi_{n}(q^2)^4}$, we get (\ref{eq;thm2_2}).
\end{proof}
\section{Generalizations of Theorems \ref{thm:1} and \ref{thm:2}}
In this section, we first give a generalization of Theorem \ref{thm:1} as follows.
 \begin{theorem}\label{thm:5}
    Let $n>1$, $d\ge 2$ , $t$ be  integers with $\gcd(t,d)=1$ and $n\equiv t \pmod {d}$ such that $n+d-nd\le t\le n$. We have
    \begin{align}
    &\sum_{k=0}^{(n-t)/d}[2dk+t] \frac{\left(q^t ; q^{d}\right)_{k}^4}{\left(q^{d }; q^{d}\right)_k^4}q^{(d-2t)k} \equiv \frac{\left(q^{2t} ; q^{d}\right)_{(n-t)/{d}}}{\left(q^{d} ; q^{d}\right)_{(n-t)/d}}[n] q^{t(t-n) / d} \quad\left(\bmod [n]\Phi_{n}(q)^{2}\right)\label{eq;thm5_1};\\
    &\sum_{k=0}^{n-1}[2dk+t] \frac{\left(q^t ; q^{d}\right)_{k}^4}{\left(q^{d }; q^{d}\right)_k^4}q^{(d-2t)k} \equiv \frac{\left(q^{2t} ; q^{d}\right)_{(n-t)/{d}}}{\left(q^{d} ; q^{d}\right)_{(n-t)/d}}[n] q^{t(t-n) / d} \quad\left(\bmod [n]\Phi_{n}(q)^{2}\right)\label{eq;thm5_2}.
    \end{align}
\end{theorem}
It is obvious that Theorem \ref{thm:1} is just the special case with $d=4$ and $t=1$ in Theorem \ref{thm:5}.
Letting $d=2$ and $t=1$ in (\ref{eq;thm5_1}), we immediately get
\begin{equation}
\sum_{k=0}^{(n-1) / 2}[4k+1] \frac{\left(q ; q^{2}\right)_{k}^4}{\left(q^2 ; q^{2}\right)_k^4} \equiv[n] q^{(1-n) / 2} \quad\left(\bmod [n]\Phi_{n}(q)^{2}\right),
\end{equation}
which is a $q$-analogue of Van Hamme's (C.2) and has been proved by Guo and Wang  \cite{GW}.
\begin{proof}
As same as the proof of Theorem \ref{thm:1},
we shall first establish the following parametric generalization of (\ref{eq;thm5_1}):
\begin{align}
&\sum_{k=0}^{(n-t)/d}[2dk+t] \frac{\left(q^t ; q^{d}\right)_{k}^2(aq^t;q^d)_k(q^t/a;q^d)_k}{\left(q^{d }; q^{d}\right)_k^2(aq^d;q^d)_k(q^d/a;q^d)_k}q^{(d-2t)k}\notag \\&\equiv \frac{\left(q^{2t} ; q^{d}\right)_{(n-t)/{d}}}{\left(q^{d} ; q^{d}\right)_{(n-t)/d}}[n] q^{t(t-n) / d} \label{eq;thm5_3}   \pmod{[n](1-aq^n)(a-q^n)}.
\end{align}
At first, the $q$-congruence (\ref{eq;thm5_3}) modulo $(1-aq^n)$ and $(a-q^n)$ follows from the summation
\begin{align}
\sum_{k=0}^{(n-t)/d}[2dk+t] \frac{\left(q^t ; q^{d}\right)_{k}^2(q^{t+n};q^d)_k(q^{t-n};q^d)_k}{\left(q^{d }; q^{d}\right)_k^2(q^{d+n};q^d)_k(q^{d-n};q^d)_k}q^{(d-2t)k}
=\frac{\left(q^{2t} ; q^{d}\right)_{(n-t)/{d}}}{\left(q^{d} ; q^{d}\right)_{(n-t)/d}}[n] q^{t(t-n) / d},
\end{align}
which is the specialization $q \mapsto q^{d}$, $a=d=q^t$, $b=q^{t-n}$ and $c=q^{t+n}$ in Rogers' nonterminating $_6\phi_5$ summation (\ref{eq:6phi5}). On the other hand, let $c_q(k)$ denotes the $k$-th term on the left-hand side of (\ref{eq;thm5_3}). Similarly to the proof of Theorem \ref{thm:3}, we can further show that
\begin{align}
\begin{aligned}
\sum_{k=0}^{\frac{n-t}{d}} c_{\zeta}(k) =\frac{1}{[t]_{\zeta}} \sum_{l=0}^{\frac{n-d s_{1}-t}{d m_{1}}-1} c_{\zeta}\left(l m_{1}\right) \sum_{k=0}^{m_{1}-1} c_{\zeta}(k)
+\frac{1}{[t]_{\zeta}} c_{\zeta}\left(\left(n-d s_{1}-t\right) / d\right) \sum_{k=0}^{s_{1}} c_{\zeta}(k)=0;
\end{aligned}
\end{align}
\begin{align}
\sum_{k=0}^{n-1} c_{\zeta}(k)=\frac{1}{[t]_\zeta} \sum_{l=0}^{n / m_{1}-1} c_{\zeta}\left(l m_{1}\right) \sum_{k=0}^{m_{1}-1} c_{\zeta}(k)=0,
\end{align}
where $\zeta \not=1 $ is a root of $\Phi_{m_1}(q)$ with $m_1 |n$, integer $s_1$ satisfies  $0\le s_1 \le m_1-1$ and $ds_1\equiv -t \pmod {m_1}$.
Then the truth of (\ref{eq;thm5_3}) modulo $[n]$ can be proved as same as the proof of (\ref{eq;integer}).
Thus we prove that (\ref{eq;thm5_3}) module $[n](1-aq^n)(a-q^n)$ is true. The $q$-supercongruences (\ref{eq;thm5_1}) and (\ref{eq;thm5_2}) then follow by letting $a\rightarrow 1$ in (\ref{eq;thm5_3}) and the fact that $(q^t;q^d)_k^4/(q^d;q^d)_k^4\equiv 0 \pmod {\Phi_{n}(q)^4}$ for $(n-t)/d <k \le n-1$.
This completes the proof.
\end{proof}

We also have the following generalization of Theorem \ref{thm:2}.

\begin{theorem}\label{thm:6}
Let $n>1$, $d\ge 2$, $t$ be  integers with $\gcd(t,d)=1$ and $n\equiv t \pmod {d}$ such that $n+d-nd\le t\le n$. Then, modulo $ [n]_{q^2} \Phi_{n}(q^2)^2$,
\begin{align}
\sum_{k=0}^{(n-t) / d}[2d k+t]_{q^{2}}[2dk+t]^{2} \frac{\left(q^{2t} ; q^{2d}\right)_{k}^{4}}{\left(q^{2d} ; q^{2d}\right)_{k}^{4}} q^{-4t k}
&\equiv \frac{-2[t]^2[n]_{q^{2}}(q^{4t};q^{2d})_{(n-t)/d}}{(1+q^{2t})(q^{2d};q^{2d})_{(n-t)/d}} q^{t-2t(n-t)/d},\label{eq;thm6_1} \\
\sum_{k=0}^{n-1}[2d k+t]_{q^{2}}[2dk+t]^{2} \frac{\left(q^{2t} ; q^{2d}\right)_{k}^{4}}{\left(q^{2d} ; q^{2d}\right)_{k}^{4}} q^{-4t k}
&\equiv \frac{-2[t]^2[n]_{q^{2}}(q^{4t};q^{2d})_{(n-t)/d}}{(1+q^{2t})(q^{2d};q^{2d})_{(n-t)/d}} q^{t-2t(n-t)/d}.\label{eq;thm6_2}
\end{align}
\end{theorem}
Obviously, the $d=4$ and $t=1$ case of this theorem reduces to Theorem \ref{thm:2}. Furthermore, letting $d=2$ and $t=1$, we get
\begin{align*}
\sum_{k=0}^{(n-1) / 2}[4 k+1]_{q^{2}}[4 k+1]^{2} \frac{\left(q^{2} ; q^{4}\right)_{k}^{4}}{\left(q^{4} ; q^{4}\right)_{k}^{4}} q^{-4 k} \equiv-[n]_{q^{2}} \frac{2 q^{2-n}}{1+q^{2}} \quad\left(\bmod [n]_{q^{2}} \Phi_{n}\left(q^{2}\right)^{2}\right),
\end{align*}
which is a $q$-analogue of (C.2) supercongruence of Van Hamme and was already obtained by Guo\cite{Guo-t}.
\begin{proof}
Letting  $q \mapsto q^{2d}$, $a=d=q^{2t}$, $b=c=q^{2d+t}$, $e=q^{2t+2n}$ and $f=q^{2t-2n}$ in  Watson's  $_8\phi_7$ transformation (\ref{eq:8phi7}),
we can prove that, modulo $(a-q^{2n})$ and $(1-aq^{2n})$,
\begin{align}
&\sum_{k=0}^{(n-t) / d}[2d k+t]_{q^{2}}[2dk+t]^{2} \frac{\left(q^{2t} ; q^{2d}\right)_{k}^{2}(aq^{2t};q^{2d})_{k}(q^{2t}/a;q^{2d})_k}{\left(q^{2d} ; q^{2d}\right)_{k}^{2}(aq^{2d};q^{2d})_{k}(q^{2d}/a;q^{2d})_k} q^{-4t k}\notag \\
&\equiv [t]^2[n]_{q^{2}}\frac{(q^{4t};q^{2d})_{(n-t)/d}}{(q^{2d};q^{2d})_{(n-t)/d}} q^{-2t(n-t)/d}\left(1-\frac{\left(1-a q^{2 t}\right)\left(1-q^{2 t} / a\right)}{\left(1-q^{t}\right)^{2}(1+ q^{2 t})}\right)\label{eq;thm6_3}.
\end{align}
In the same manner as the proof of Theorem \ref{thm:3}, we can show that
\begin{align*}
\lim_{q \rightarrow \eta}\sum_{k=0}^{(n-t)/d} p_{q}(k)=\lim_{q \rightarrow \eta}\sum_{k=0}^{n-1} p_{q}(k)=0,
\end{align*}
where $p_q(k)$ is the  $k$-th term on the left-hand side of (\ref{eq;thm6_3}) and $\eta \not= \pm 1$ is a root of $\Phi_{m_3}(q^2)$ with $m_3|n$ and $m_3\geq1$.
This proves that (\ref{eq;thm6_3}) is true modulo $[n]_{q^2}(a-q^{2n})(1-aq^{2n})$. The rest of the proof
is similar to that of Theorem \ref{thm:2} and is omitted here.
\end{proof}
\section{A Conjecture about Swisher's (G.3)}
In the last part of Swisher's \cite{Swisher} paper, he conjectured a series of general congruences about Van Hamme's first 12 supercongruences, which are deemed to Dwork-type congruences, such as (G.3), for $p\equiv 1 \pmod 4$,
\begin{align}
\sum_{k=0}^{(p^r-1)/4}(8k+1)\frac{(\frac{1}{4})_k^4}{k!^4}
\equiv -(-1)^{\frac{p^2-1}{8}}p\Gamma_{p}(\frac{1}{2})\Gamma_{p}(\frac{1}{4})^2
 \sum_{k=0}^{(p^{r-1}-1)/4}(8k+1)\frac{(\frac{1}{4})_k^4}{k!^4} \pmod{p^{4r}} \label{eq;Swisher}.
\end{align}

Note that $ (-1)^{\frac{p^2-1}{8}}=(-1)^{\frac{p-1}{4}}$  and $\Gamma_{p}(\frac{1}{4})\Gamma_{p}(\frac{3}{4})=-(-1)^{\frac{p-1}{4}}$ for $p\equiv 1 \pmod 4$, the right-side hand of (\ref{eq;Swisher}) can be written as
\begin{align*}
 p\frac{\Gamma_p(\frac 12)\Gamma_p(\frac 14)}{\Gamma_p(\frac 34)}
 \sum_{k=0}^{(p^{r-1}-1)/2}(8k+1)\frac{(\frac{1}{4})_k^4}{k!^4} .
\end{align*}

Not long ago, Guo\cite{Guo-jmaa} and Zudilin\cite{Guo21} proved a number of Dwork-type supercongruences, including (B.3) and some special cases of (C.3), (E.3) and (F.3) in \cite{Swisher}, by constructing suitable $q$-analogues. We now propose the partial $q$-analogues of (G.3). 
It should be pointed out that the machinery in \cite{Guo-jmaa,Guo21} does not work for these q-congruences.
\begin{conjecture}
Let $r>1$, $n>1$ be integers with $n\equiv 1 \pmod 4$. Then, modulo $[n^r]\prod_{j=1}^{r}\Phi_{n^j}(q)^2$, we have
\begin{align}
\sum_{k=0}^{(n^r-1) / 4}[8k+1] \frac{\left(q ; q^{4}\right)_{k}^4}{\left(q^4 ; q^{4}\right)_k^4}q^{2k} &\equiv \frac{\left(q^{2} ; \label{con-1} q^{4}\right)_{(n^r-1)/{4}}}{\left(q^{4} ; q^{4}\right)_{(n^r-1)/4}}\frac{ \left(q^{4n} ; q^{4n}\right)_{(n^{r-1}-1)/{4}} }{\left(q^{2n} ; q^{4n}\right)_{(n^{r-1}-1)/{4}}}[n] q^{(1-n) / 4}\notag \\
&\times \sum_{k=0}^{(n^{r-1}-1) / 4}[8k+1]_{q^n} \frac{\left(q^n ; q^{4n}\right)_{k}^4}{\left(q^{4n} ; q^{4n}\right)_k^4}q^{2nk},\\
\sum_{k=0}^{n^r-1}[8k+1] \frac{\left(q ; q^{4}\right)_{k}^4}{\left(q^4 ; q^{4}\right)_k^4}q^{2k} &\equiv \frac{\left(q^{2} ; q^{4}\right)_{(n^r-1)/{4}}}{\left(q^{4} ; q^{4}\right)_{(n^r-1)/4}}\frac{ \left(q^{4n} ; q^{4n}\right)_{(n^{r-1}-1)/{4}} }{\left(q^{2n} ; q^{4n}\right)_{(n^{r-1}-1)/{4}}}[n] q^{(1-n) / 4}\notag \\
&\times \sum_{k=0}^{n^{r-1}-1}[8k+1]_{q^n} \frac{\left(q^n ; q^{4n}\right)_{k}^4}{\left(q^{4n} ; q^{4n}\right)_k^4}q^{2nk}.\label{G.3-2}
\end{align}
\end{conjecture}
 Letting $n=p$ and $q\rightarrow 1$ in \eqref{con-1}, we immediately get
\begin{align*}
\sum_{k=0}^{(p^r-1) / 4}(8k+1) \frac{(\frac{1}{4})_k^4}{k!^4} &\equiv \frac{ (\frac{1}{2})_{(p^r-1)/4} (1)_{(p^{r-1}-1) / 4}     }{ (1)_{(p^{r}-1) / 4 }(\frac{1}{2})_{(p^{r-1}-1)/4} }\:p\notag  \sum_{k=0}^{(p^{r-1}-1)/4}(8k+1)\frac{(\frac{1}{4})_k^4}{k!^4} \pmod{3^r}.
 \end{align*}
In order to prove that  \eqref{con-1} is a direct $q$-analogue of (G.3) modulo $p^{3r}$, we only need to verify that
\begin{align*}
\frac{ (\frac{1}{2})_{(p^r-1)/4} (1)_{(p^{r-1}-1) / 4}}{ (1)_{(p^{r}-1) / 4 }(\frac{1}{2})_{(p^{r-1}-1)/4}}\equiv \frac{\Gamma_p(\frac 12)\Gamma_p(\frac 14)}{\Gamma_p(\frac 34)} \pmod{p^{2r}}.
\end{align*}
It is obvious that \eqref{G.3-2} is an equivalent form of \eqref{con-1}.

\end{document}